\documentclass[12pt, reqno]{amsart}
\usepackage{amscd,amsmath,amsthm,amssymb,graphics}
\usepackage{amsfonts,amssymb,amscd,amsmath,enumitem,verbatim}
\usepackage[a4paper,top=3cm,left=3cm,right=3cm]{geometry}
\theoremstyle{plain}
\usepackage{hyperref}
\newtheorem{Theorem}{Theorem}
\newtheorem{Lemma}[Theorem]{Lemma}
\newtheorem{Corollary}[Theorem]{Corollary}
\newtheorem{Proposition}[Theorem]{Proposition}

\newtheorem{Conjecture}[Theorem]{Conjecture}

\theoremstyle{definition}

\usepackage{graphicx}
\usepackage{float}
\usepackage{xfrac}  
\usepackage{faktor}

\title{On alternative definition of Lucas atoms and their $p$-adic valuations}

\author{Gessica Alecci, Piotr Miska, Nadir Murru, Giuliano Romeo}

\begin{document}

\maketitle

\begin{abstract}
Lucas atoms are irreducible factors of Lucas polynomials and they were introduced in \cite{ST}. The main aim of the authors was to investigate, from an innovatory point of view, when some combinatorial rational functions are actually polynomials.
In this paper, we see that the Lucas atoms can be introduced in a more natural and powerful way than the original definition, providing straightforward proofs for their main properties. Moreover, we fully characterize the $p$-adic valuations of Lucas atoms for any prime $p$, answering to a problem left open in \cite{ST}, where the authors treated only some specific cases for $p \in \{2, 3\}$. Finally, we prove that the sequence of Lucas atoms is not holonomic, contrarily to the Lucas sequence that is a linear recurrent sequence of order two.
\end{abstract}

\section{Introduction}

Lucas sequences of the first kind $(U_n)_{n \geq 0}$ are linear recurrent sequences having characteristic polynomial $X^2 - sX - t$ with initial conditions $0$ and $1$, i.e., they are defined by the recurrence 
\[\begin{cases} U_0 = 0,\quad U_1 = 1 \cr U_n = s U_{n-1} + t U_{n-2}, \quad \forall n \geq 2,\end{cases}\]
where $s$ and $t$ are usually integer numbers. If we consider $s$ and $t$ as two variables, then we talk about Lucas polynomials $U_n(s,t) \in \mathbb N[s, t]$. In \cite{ST}, the authors studied a very interesting factorization for Lucas polynomials connected to cyclotomic polynomials. In particular, they introduced the \emph{Lucas atoms} as the polynomials 
\[P_1(s,t) := 1,\quad P_{n}(s,t) := \Gamma(\Psi_n(q)),\]
for all $n \geq 2$, where $\Psi_n(q)$ is the $n$--th cyclotomic polynomial and $\Gamma$ a map that exploits the gamma expansion of palindromic polynomials.
Given this definition, the authors proved that the following factorization of the Lucas polynomials holds:
\[ U_n(s,t) = \prod_{d \mid  n} P_d(s,t),\]
and moreover $P_n(s,t) \in \mathbb N[s,t]$, for all $n \geq 1$.

This study was firstly motivated by the problem of finding when the rational function
\begin{equation}\label{eq:lanalogue} \cfrac{\prod_i U_{n_i}(s,t)}{\prod_j U_{k_j}(s,t)} \end{equation}
is actually a polynomial. For instance, it has been studied by some authors the case of the so called Lucanomial, which is the generalization of the binomial coefficient to Lucas polynomials:
\[ \binom{U_n(s,t)}{U_k(s,t)} := \cfrac{\prod_{i=1}^nU_i(s,t)}{\prod_{i=1}^kU_i(s,t)\prod_{i=1}^{n-k}U_i(s,t)}, \]
see, e.g., \cite{BP, BCMS}. In fact, thanks to Lucas atoms, the study of when \eqref{eq:lanalogue} is a polynomial can be approached in a straightforward way exploiting the factorization of Lucas polynomials.

The definition of Lucas atoms can be simplified, avoiding the use of the $\Gamma$ map. The idea of factorizing the Lucas polynomials dates back to 1969, when Webb and Parberry \cite{WP} employed it to discuss the irreducibility of the Lucas polynomial $U_n(s,1)$. Subsequently, Levy \cite{Levy} gave the definition of fibotomic polynomials, which turn out to be the Lucas atoms for $t = \pm 1$ and he proved that they are irreducible. Moreover, he made some remarks on their connection with the two-variable homogeneous cyclotomic polynomials, which was already highlighted in a work of Brillhart et al. \cite{BMS}. It is also remarkable that this approach has been already used by Stewart et al. \cite{SS, STE0, STE1, STE2, STE3} in order to obtain estimations on the greatest prime factor of the terms of the Lucas sequence and other reccurent sequences.

Let us consider
\[\Phi_n(\alpha, \beta) := \prod_{\substack{j=1 \\ (j,n)=1}}^n (\alpha - \omega^j \beta),\]
for all $n \geq 1$, where $\omega$ is an $n$--th primitive root of unity. For basic properties of these polynomials we refer to \cite{STE1}.
From this definition, we immediately get that 
\[\alpha^n - \beta^n = \prod_{d\mid n} \Phi_d(\alpha, \beta),\]
and
\[\beta^{\varphi(n)} \Psi_n(\alpha/\beta) = \Phi_n(\alpha, \beta),\]
where $\varphi(\cdot)$ is the Euler's totient function. Then, we can define the Lucas atoms as the polynomials
\begin{equation} \label{eq:def-la}P_1(s,t) := 1, \quad P_n(s,t) := \Phi_n(\alpha, \beta) = \beta^{\varphi(n)} \Psi_n(\alpha/\beta),\end{equation}
for all $n \geq 2$, where $s = \alpha + \beta$ and $t = -\alpha\beta$. In this way, we obtain the factorization of the Lucas polynomials by means of the Lucas atoms (as well as with the definition given in \cite{ST}). Indeed, observing that $\Phi_1(\alpha, \beta) = \alpha - \beta$, we have
\begin{equation*}
\alpha^n - \beta^n = (\alpha - \beta) \prod_{\substack{d\mid n \\ d\not=1}}\Phi_d(\alpha, \beta) = (\alpha - \beta) \prod_{d\mid n} P_d(s,t),
\end{equation*}
remembering that $P_1(s,t) = 1$. Thus, we have
\begin{equation} \label{eq:fatt}U_n(s,t) = \cfrac{\alpha^n - \beta^n}{\alpha - \beta} = \prod_{d\mid n} P_d(s,t)\end{equation}
where $(U_n(s,t))_{n \geq 0}$ has characteristic polynomial $X^2 - sX - t = (X - \alpha)(X - \beta)$. 

In this paper, first we highlight that the definition of Lucas atoms given in \eqref{eq:def-la} is more convenient and straightforward than the original definition proposed in \cite{ST}. Specifically, in Section \ref{sec:rev} we revisit some of the main properties of Lucas atoms, obtaining them with elementary proofs. Section \ref{sec:padic} is devoted to the $p$-adic valuations of Lucas atoms. In general, the $p$-adic valuations for integer sequences is a well studied topic. In particular, the case of Lucas sequences has been deepened by several authors (see, e.g., \cite{BAL, LEN, SAN, WARD}). In \cite{ST}, the authors dealt with Lucas atoms and some divisibility properties by $p = 2, 3$. They left open, addressing it as a hard problem, the extension of these results to arbitrary primes. In Section \ref{sec:padic}, we solve this problem and we completely characterize the $p$-adic valuations of Lucas atoms. Finally, in Section \ref{sec:nonrec}, we exploit the results on the $p$-adic valuations of Lucas atoms to prove that the sequence of Lucas atoms is not holonomic, i.e., it does not satisfy any recurrence relation, also considering coefficients being polynomials, contrarily to the Lucas sequence which is a binary linear recurrent sequence.

\section{Revisiting some properties of Lucas atoms via cyclotomic polynomials} \label{sec:rev}

In this section, we obtain some properties of Lucas atoms exploiting the definition \eqref{eq:def-la}.  Here, we will always consider $s, t, \alpha$ and $\beta$ as variables related by
\[ s = \alpha + \beta, \quad t = -\alpha\beta. \]
First of all, we prove that the Lucas atoms are actually polynomials with natural coefficients.

\begin{Lemma} \label{lemma:ab}
For all $n\geq 0$, we have $\alpha^n + \beta^n \in \mathbb N[s,t]$. 
\end{Lemma}
\begin{proof}
We prove the Lemma by induction. The first steps are straightforward:
\[\alpha + \beta = s, \quad \alpha^2 + \beta^2 = (\alpha+\beta)^2 - 2\alpha\beta = s^2 + 2t.\]
Now, consider an integer $n > 2$ and suppose $\alpha^i + \beta^i \in \mathbb N[s,t]$ for all $i \leq n-1$. If $n$ is odd, then
\begin{align*}
\alpha^n + \beta^n &= (\alpha + \beta)\sum_{i=0}^{n-1}(-1)^i\alpha^{n-1-i}\beta^i \cr
&= (\alpha + \beta)\Big(\alpha^{n-1} + \beta^{n-1} +(-\alpha \beta)^{(n-1)/2} + \sum_{i=1}^{(n-3)/2}(-\alpha\beta)^i(\alpha^{n-1-2i} + \beta^{n-1-2i})\Big) \cr
&= s\Big(\alpha^{n-1} + \beta^{n-1} + t^{(n-1)/2} + \sum_{i=1}^{(n-3)/2}t^i(\alpha^{n-1-2i} + \beta^{n-1-2i})\Big),
\end{align*}
and by the inductive hypothesis we have $\alpha^n + \beta^n \in \mathbb N[s,t]$. If $n = 2^k h$, with $k>0$ and $h>1$ odd, then we can write
\[ \alpha^n + \beta^n = (\alpha^{2^k} + \beta^{2^k})\sum_{i=0}^{h-1}(-1)^i(\alpha^{2^k})^{h-1-i}(\beta^{2^k})^i, \]
and the thesis follows as above.
Finally, if $n = 2^k$, by induction we can prove that we always have the term $2t^{2^{k-1}}$ in $\alpha^{2^k} + \beta^{2^k}$, indeed we can write
\[ \alpha^{2^k} + \beta^{2^k} = (\alpha^{2^{k-1}} + \beta^{2^{k-1}})^2 - 2 \alpha^{2^{k-1}}\beta^{2^{k-1}}, \]
and consequently we also have $\alpha^{2^k} + \beta^{2^k} \in \mathbb N[s,t]$.
\end{proof}

\begin{Proposition}
For all $n \geq 1$, we have $P_n(s,t) \in \mathbb N[s, t]$.
\end{Proposition}
\begin{proof}
From Lemma \ref{lemma:ab}, we have that $P_n(s,t) \in \mathbb Z[s, t]$, for all $n \geq 1$. Indeed, $P_n(s,t) = \beta^{\varphi(n)} \Psi_n(\alpha/\beta)$ and by the palindromicity of the cyclotomic polynomials, we have
\[ P_n(s,t) = \alpha^{\varphi(n)} + \beta^{\varphi(n)} + \sum_{j=1}^{(\varphi(n)-2)/2}c_j (\alpha \beta)^{i_j}(\alpha^{k_j} + \beta^{k_j}),\]
where $c_j \in \mathbb Z$ and $i_j, k_j \in \mathbb N$. Moreover, we can observe that
\[P_n(s,t) = \Phi_n(\alpha, \beta) = \prod_{j=1}^{\varphi(n)/2}(\alpha^2 + \beta^2 - (\omega_j + \bar \omega_j)\alpha\beta)\]
where $\omega_j$ is a primitive $n$-th root of unity and $\bar \omega_j$ is its complex conjugate. Each factor in the above formula can be also written as
\[ (\alpha + \beta)^2 - (\omega_j + \bar\omega_j + 2)\alpha \beta = s^2 +(\omega_j + \bar\omega_j + 2)t, \]
where $\omega_j + \bar\omega_j + 2 > 0$ for all $j$. This ensures that the coefficients of $P_n(s,t)$ must be non-negative integers.
\end{proof}

In \cite{Levy} Levy proved the irreducibility of the univariate Lucas atoms $P_n(s,\pm 1)$. Exploiting the connection with cyclotomic polynomials $\Phi_n(\alpha,\beta)$, we are able to provide the same result for general $t$, i.e. for the Lucas atoms $P_n(s,t)$ in two variables.

\begin{Proposition}\label{Pro: irred}
The Lucas atoms $P_n(s,t)$ are irreducible polynomials over $\mathbb Q$, for all $n\in\mathbb{N}$.
\end{Proposition}
\begin{proof}
Note that $P_n(s,t)=P_n(\alpha +\beta ,-\alpha\beta )=\Phi_n(\alpha ,\beta )$. The polynomial $\Phi_n$ is irreducible over $\mathbb Q$ as the homogenization of $\Psi_n$, which is irreducible over $\mathbb Q$ as well. Then any factorization of $P_n(s,t)$ into a product of two polynomials $A_n(s,t), B_n(s,t)\in\mathbb Q[s,t]$ can be rewritten as
$$\Phi_n (\alpha ,\beta )=P_n(\alpha +\beta ,-\alpha\beta )=A_n(\alpha +\beta ,-\alpha\beta )B_n(\alpha +\beta ,-\alpha\beta ).$$
Because of the irreducibility of $\Phi_n (\alpha ,\beta )$ over $\mathbb Q$ we conclude that either $A_n(s,t)=A_n(\alpha +\beta ,-\alpha\beta )$ or $B_n(s,t)=B_n(\alpha +\beta ,-\alpha\beta )$ is a constant polynomial. This proves the irreducibility of $P_n(s,t)$ over $\mathbb Q$.
\end{proof}

From Proposition \ref{Pro: irred}, using the irreducibility of the Lucas atoms $P_n(s,t)$, we obtain a simple proof of the following theorem, which is one of the main results of \cite{ST}.

\begin{Theorem}
Let us suppose $f(s,t)=\prod\limits_{d\geq 2} U_{n_i}(s,t)$ and $g(s,t)=\prod\limits_{d\geq 2} U_{k_j}(s,t)$, for $n_i,k_j\in\mathbb{N}$, and write their atomic decompositions as
\[f(s,t)=\prod\limits_{d\geq 2} P_{d}(s,t)^{a_d},\ \ g(s,t)=\prod\limits_{d\geq 2} P_{d}(s,t)^{b_d},\]
for $a_d,b_d\in\mathbb{N}$. Then $f(s,t)/g(s,t)$ is a polynomial if and only if $a_d\geq b_d$ for all $d\geq 2$. Furthermore in this case $f(s,t)/g(s,t)$ has non-negative integer coefficients.
\begin{proof}
The condition $a_d\geq b_d$ is clearly sufficient for $f(s,t)/g(s,t)$ being a polynomial. Conversely, since the polynomials $P_d(s,t)$ are irreducible, by Proposition \ref{Pro: irred}, the Lucas atoms at the denominator cancel out only if they are present at the numerator with a greater or equal exponent. Moreover, the ratio $f(s,t)/g(s,t)$ has nonnegative integer coefficients because it is the product of the remaining Lucas atoms.
\end{proof}
\end{Theorem}

Lucas formula and Gauss formula for cyclotomic polynomials can be easily adapted to Lucas atoms. For instance, from the Lucas formula we have that, if $n \geq 5$ odd and squarefree, then there exist two palindromic polynomials $C_n(\alpha/\beta), D_n(\alpha/\beta) \in \mathbb Z[\alpha/\beta]$, with degrees $\varphi(n)/2$ and $\varphi(n)/2 - 1$, respectively, such that
\[ \Psi_n((-1)^{(n-1)/2}\alpha/\beta) = C_n^2(\alpha/\beta) - n \frac{\alpha}{\beta}D_n^2(\alpha/\beta). \]
Thus, if $n \equiv 1 \pmod 4$, we obtain 
\[\beta^{\varphi(n)}\Psi_n(\alpha/\beta) = \beta^{\varphi(n)}C_n^2(\alpha/\beta) - n \alpha \beta^{\varphi(n)-1} D_n^2(\alpha/\beta) = \tilde C_n^2(\alpha, \beta) - n \alpha\beta \tilde D_n^2(\alpha, \beta),\]
where $\tilde C_n(\alpha, \beta), \tilde D_n(\alpha, \beta) \in \mathbb Z[\alpha, \beta]$ are palindromic with degrees $\varphi(n)/2$ and $\varphi(n)/2 - 1$, respectively. Then, by palindromicity of these polynomials and by Lemma \ref{lemma:ab} we get
\[ P_n(s,t) = F_n^2(s,t) + n t G_n^2(s,t) \]
with $F_n(s,t), G_n(s,t) \in \mathbb Z[s,t]$. Similar results can be obtained for the case $n$ even and for the Gauss formula. The same results were obtained in a simple way in \cite{ST}, but with the definition \eqref{eq:def-la} of the Lucas atoms, it becomes more clear the fact that the Lucas formula, for $n$ odd, can not be adapted when $n \equiv 3 \pmod 4$. Indeed, in this case we would have
\[ \beta^{\varphi(n)}\Psi_n(-\alpha/\beta) = \beta^{\varphi(n)}C_n^2(\alpha/\beta) - n \alpha \beta^{\varphi(n)} D_n^2(\alpha/\beta) \]
where on the left side we have the Lucas atoms $P_n(s,t)$ but where $s = -\alpha + \beta$ and $t = \alpha\beta$ making not possible to obtain on the right hand some polynomials with these variables $s$ and $t$ and integer coefficients.

In \cite{ST}, the authors proved also an analogue of a reduction formula for cyclotomic polynomials. In order to provide a proof of this result, the authors proved several combinatorial lemmas, claiming that a proof could not be found easily and directly from the connection with cyclotomic polynomials provided by the function $\Gamma$. Here we show that, exploiting \eqref{eq:def-la}, the proof becomes straightforward.

\begin{Theorem}\label{Thm: ST}
If $n\geq 2$ is a positive integer and $p$ is a prime not dividing $n$, then
\begin{align*}
P_{pn}(s,t)=\begin{cases} \dfrac{P_n(s^2+2t,-t^2)}{P_n(s,t)}\ \  &\text{if } p=2 \\
\dfrac{P_n(sP_{2p},t^p)}{P_n(s,t)} &\text{if } p\geq 3,
\end{cases}
\end{align*}
where writing $P_m$ we mean $P_m(s,t)$.
\end{Theorem}

\begin{proof}
By \eqref{eq:def-la} and the reduction formulas for cyclotomic polynomials,
\begin{equation}\label{Eq: reduction}
P_{pn}(s,t)=\Phi_{pn}(\alpha,\beta)=\frac{\Phi_{n}(\alpha^p,\beta^p)}{\Phi_{n}(\alpha,\beta)}=\frac{\Phi_{n}(\alpha^p,\beta^p)}{P_n(s,t)}.
\end{equation}
Now let us notice that $\alpha^p$ and $\beta^p$ are the roots of the polynomial
\[(X-\alpha^p)(X-\beta^p)=X^2-(\alpha^p+\beta^p)X+(\alpha\beta)^p,\]
hence the Lucas atom correspondent to $\Phi_{n}(\alpha^p,\beta^p)$ is
\[P_n(\alpha^p+\beta^p,-(\alpha\beta)^p)=P_n(\alpha^p+\beta^p,-(-t)^p).\]
If $p=2$,
\[\Phi_{n}(\alpha^2,\beta^2)=P_n(\alpha^2+\beta^2,-t^2)=P_n(s^2+2t,-t^2),\]
and this concludes the proof for this case. For $p\geq 3$, let us notice that
\[U_{2p}(s,t)=\frac{\alpha^{2p}-\beta^{2p}}{\alpha-\beta}=\frac{\alpha^p-\beta^p}{\alpha-\beta} (\alpha^p+\beta^p)=P_p(s,t)(\alpha^p+\beta^p).\]
By the definition of Lucas atoms, this means that
\[\alpha^p+\beta^p=P_2P_{2p}=(\alpha+\beta)P_{2p}=sP_{2p}.\]
Therefore,
\[\Phi_{n}(\alpha^p,\beta^p)=P_n(\alpha^p+\beta^p,-(-t)^p)=P_n(sP_{2p},t^p),\]
and also the case of odd $p$ is complete.
\end{proof}

Using a similar argument, it is easily obtained also the following theorem, which is the other main result of Section 5 in \cite{ST}.

\begin{Theorem}\label{Thm: ST2}
If $n\geq 2$ is a positive integer and $p$ is a prime not dividing $n$, then,
\begin{align*}
P_{p^m n}(s,t)=\begin{cases} P_{p^{m-1}n}(s^2+2t,-t^2)\ \  &\text{if } p=2  \\
P_{p^{m-1n}}(sP_{2p},t^p) &\text{if } p\geq 3,
\end{cases}
\end{align*}
for all $m\geq 2$.
\end{Theorem}

If $\Phi_n(a,b)=p$ for some index $n$, $p$ prime and $a, b$ integers, then 
\[P_n(a+b,-ab)=\Phi_n(a,b)=p,\]
where $a+b$ and $-ab$ are still integers. For $b=1$, a famous conjecture of Bunyakovsky implies that, for a fixed $n$, $\Phi_n(a,1)=\Psi_n(a)$ is prime for infinitely many positive integers $a$. In light of this, we can state the following conjecture which is weaker than the Bunyakovsky one.

\begin{Conjecture}
For each integer $n\geq 2$ there exist infinitely many pairs of integers $(s,t)$ such that $P_n(s,t)$ is prime. 
\end{Conjecture}

It is easy to see that the polynomials $P_2(s,t)=s$, $P_3(s,t)=s^2+t$ and $P_4(s,t)=s^2+2t$ represent all the integers, in particular all the prime numbers. The polynomial $P_6(s,t)=s^2+3t$, meanwhile, represents all the integers not congruent to $2$ modulo $3$, in particular all the prime numbers of this form. For remaining polynomials $P_n(s,t)$, whose degrees are at least equal to $4$, we do not know any tool for proving that they represent infinitely many prime numbers.

\section{$p$--adic valuations of Lucas atoms} \label{sec:padic}

In this section, we fully characterize the $p$-adic valuation of Lucas atoms. In this way we solve a problem left open in \cite{ST}, which the authors addressed as hard. In particular, they treated only some cases for $p \in \{2, 3\}$ (see \cite[Theorems 6.3, 6.5]{ST}), leaving open the general problem of extending their results to arbitrary primes.

In the following, we consider $s, t$ as integers and $\alpha, \beta$ as the roots of the polynomial $X^2 - sX - t$ and we will denote by $\Delta$ its discriminant. 

Given an integer $n\neq 0$, let $\rho(n, U)$ be the rank of appearance of $n$ in the sequence $(U_m)_{m \geq 0}$, i.e., the minimum positive integer $k$ such that $n\mid U_k$. The next results are useful known properties of the rank of appearance (see, e.g., \cite{RIB}). We give the proofs for completeness.

\begin{Lemma}\label{Lem: PropRank}
Given an integer $n\neq 0$, if $\gcd(n, t) = 1$, we have that $\rho(n, U) \mid  m$ if and only if $n \mid  U_m$.
\end{Lemma}
\begin{proof}
Let $f(X) = X^2 - sX - t$ be the characteristic polynomial of $(U_k)_{k\geq0}$ and we also define the linear recurrent sequence $(T_k)_{k \geq 0}$ with characteristic polynomial $f(X)$ and initial conditions $(1,0)$. Furthermore, define the ring $R:= \faktor{\mathbb Z_n[X]}{(f(X))}$ and the group $G:= \faktor{R^*}{\mathbb Z_n^*}$. We denote by $[a + bX]$ the elements of $G$.

In $R$, we can prove by induction that
\[ X^m = T_m + U_m X \]
for all $m \geq 0$. The base cases hold:
\[ X^0 = T_0 + U_0 X , \quad X^1 = T_1 + U_1X. \]
Then, 
\[X^m = X^{m-1}X = T_{m-1}X + U_{m-1}X^2 = T_{m-1}X + U_{m-1}(sX+t).\]
Considering that $T_m = tU_{m-1}$, for all $m \geq 2$, we obtain
\[X^m = t U_{m-1} + (sU_{m-1}+T_{m-1})X = T_m + (sU_{m-1} + tU_{m-2})X = T_m + U_mX.\]
Since $\gcd(n, t) = 1$, we can observe that $X \in R^*$ and we prove now that the order of $[X]$ in $G$ is $\rho(n, U)$. Indeed,
\begin{align*}
    \text{ord}_G [X] &= \min \{ k \in \mathbb N_+ : [X]^k = 1 \} = \min\{ k \in \mathbb N_+ : X^k \in \mathbb Z_n^*\}\\
    &= \min\{ k \in \mathbb N_+: T_k + U_k X \in \mathbb Z_n^* \} = \min\{ k \in \mathbb N_+: U_k \equiv 0 \pmod n \} \\
    &= \min\{ k \in \mathbb N_+: n \mid  U_k \} = \rho(n, U).
\end{align*}
By the same reasoning we see that $n \mid  U_m$ if and only if $[X]^m=1$. By the property of the order of an element in a group we know that $[X]^m=1$ exactly when $\text{ord}_G[X]=\rho(n,U)\mid m$. The equality $\text{ord}_G[X]=\rho(n,U)$ also proves that $\rho(n, U)$ exists for all $n$ under the hypothesis of the lemma.
\end{proof}

\begin{Lemma} \label{lemma:leg}
Given a prime number $p$ such that $p \nmid t$, the rank of appearance $\rho(p, U)$ divides $p-\left(\frac{\Delta}{p}\right)$, where $\left(\frac{\Delta}{p}\right)$ is the Legendre symbol. 
\end{Lemma}
\begin{proof}
If $p \mid  \Delta$, then $\alpha \equiv \beta \pmod p$ and $U_p = \sum_{j=0}^{p-1} \alpha^{p-1-j}\beta^j \equiv p\alpha^{p-1} \equiv 0 \pmod p$, i.e., $\rho(p, U) = p$.
Considering
\[L = \begin{pmatrix} s &  t \cr 1 & 0 \end{pmatrix},\]
we have that 
\[L^n\begin{pmatrix} 1 \cr 0 \end{pmatrix} = \begin{pmatrix} U_{n+1} \cr U_n  \end{pmatrix}\]
for all $n \geq 0$. For $p \nmid \Delta$, the matrix $L$ is similar to the diagonal matrix
\[\begin{pmatrix} \alpha &  0 \cr 0 & \beta \end{pmatrix}.\]
If $\left(\frac{\Delta}{p}\right) = -1$, then by the Frobenius morphism, we have 
\[ \begin{pmatrix} U_{p+2} \cr U_{p+1}  \end{pmatrix} = L^p \cdot L \begin{pmatrix} 1 \cr 0 \end{pmatrix} \equiv \begin{pmatrix} \alpha\beta &  0 \cr 0 & \alpha\beta \end{pmatrix}\begin{pmatrix} 1 \cr 0 \end{pmatrix} = -t \begin{pmatrix} 1 \cr 0 \end{pmatrix} \pmod p \]
as $\alpha ,\beta\in\mathbb F_{p^2}\backslash\mathbb F_p$ and $L^p$ is similar to the matrix
\[\begin{pmatrix} \alpha^p &  0 \cr 0 & \beta^p \end{pmatrix}\equiv\begin{pmatrix} \beta &  0 \cr 0 & \alpha \end{pmatrix}\]
via the same similarity matrix as $L$ and $\begin{pmatrix} \alpha &  0 \cr 0 & \beta \end{pmatrix}$.

If $\left(\frac{\Delta}{p}\right) = 1$, then by the Fermat's little theorem, we have
\[ \begin{pmatrix} U_{p} \cr U_{p-1}  \end{pmatrix} = L^{p-1} \begin{pmatrix} 1 \cr 0 \end{pmatrix} \equiv \begin{pmatrix} 1 \cr 0 \end{pmatrix} \pmod p \]
as $\alpha ,\beta\in\mathbb F_p$.
\end{proof}

Similarly to the rank of appearance of a non-zero integer in the Lucas sequence, we will denote by $\rho(n, P)$ the rank of appearance of the integer $n\neq 0$ in the sequence of Lucas atoms $(P_m)_{m \geq 1}$.
In the following lemma we prove that the rank of appearance of a given prime number in a Lucas sequence is the same for the corresponding sequence of Lucas atoms.

\begin{Lemma}\label{Lem: RankP}
Given a prime $p$, we have
\[\rho(p, P)=\rho(p, U) = k\]
and
\[v_p(U_{k}) = v_p(P_k),\]
where $v_p(\cdot)$ denotes the $p$-adic valuation.
\begin{proof}
Consider $k = \rho(n, U)$. By the definition of rank of appearance we have that $p \mid  U_k$ and $p \nmid U_d$ for any $d<k$. If $p$ divides $P_d$ for some $d<k$, then it divides $U_d$ and this is a contradiction. Therefore, the rank of appearance of $p$ in the sequence of Lucas atoms must be greater than or equal to $k$. Moreover,
\[v_p(U_k)=\sum\limits_{d\mid k}v_p(P_d)=v_p(P_k),\]
so that $\rho(p, P)= k$ and $v_p(U_k)=v_p(P_k)$.
\end{proof}
\end{Lemma}

For studying the $p$-adic valuations of Lucas atoms, we use the following results of Ballot \cite{BAL, BALerr} and Sanna \cite{SAN} on the $p$--adic valuations of Lucas sequences.

\begin{Theorem}[\cite{SAN}, Corollary 1.6]\label{Thm: Carlo}
Let $p\geq 3$ be a prime number such that $p\nmid t$ and $k=\rho(p, U)$. Then,
\begin{align*}
v_p(U_n)=\begin{cases} v_p(n) + v_p(U_p) - 1 \ \  &\text{if} \ \  p\mid \Delta, \ p\mid  n \\
0 \ \  &\text{if} \ \  p\mid \Delta, \ p\nmid n \\
v_p(n) + v_p(U_k) \ \  &\text{if} \ \  p\nmid\Delta, \ k\mid  n \\
0 \ \  &\text{if} \ \  p\nmid\Delta, \ k\nmid n,
\end{cases}
\end{align*}
for each positive integer $n$, where $\nu_p(U_p)=1$ for $p\geq 5$ if $p\mid\Delta$.
\end{Theorem}

\begin{Theorem}[\cite{SAN}, Theorem 1.5 for $p=2$]\label{Thm: Carlo2}
If $2 \nmid t$ and $2\mid s$ (i.e., $2\mid \Delta$ and $\rho(2,U) = 2$), then
\[v_2(U_n) = 
\begin{cases}
v_2(n) + v_2(U_2) - 1 \ \  &\text{if} \ \  2\mid n \cr
0 \ \  &\text{if} \ \  2\nmid n.
\end{cases}
\]
If $2 \nmid t$ and $2\nmid s$ (i.e., $2\nmid\Delta$ and $\rho(2,U) = 3$), then
\[v_2(U_n) = 
\begin{cases}
v_2(n) + v_2(U_6) - 1 \ \  &\text{if} \ \  3\mid n, \ 2\mid n \cr
v_2(U_3) \ \  &\text{if} \ \ 3\mid n, \ 2 \nmid n \cr
0 \ \  &\text{if} \ \  3\nmid n.
\end{cases}
\]
\end{Theorem}

\begin{Theorem}[\cite{BAL}, Theorem 1.2]\label{Thm: Ballot}
Let $p$ be a prime such that $s = p^a s'$ and $t = p^b t'$, where $p \nmid s' t'$ and $a,b\in\mathbb N_+\cup\{\infty\}$ with $a,b$ necessarily finite in the case of $b=2a$. Then for all $n \geq 1$, we have
\[ v_p(U_n) = \begin{cases}  (n-1)a \ \ &\text{if } \ \ b > 2a \cr (n-1)a + v_p(U'_n) \ \ &\text{if } \ \ b = 2a, \end{cases} \]
where $(U'_n)_{n \geq 0}$ is the Lucas sequence with characteristic polynomial $X^2 - s' X - t'$. For $b < 2a$, we have that
\[ \begin{cases}v_p(U_{2n}) =  bn + (a-b) + v_p(n) + \lambda_n \cr v_p(U_{2n+1}) = bn,    \end{cases} \]
where
\[ \lambda_n = \begin{cases} v_p(s'^2 -t') \ \ &\text{if } 2 \leq p \leq 3, \ 2a = b + 1, \ p\mid n \cr 0 \ \ &\text{otherwise}.  \end{cases} \]
\end{Theorem}

In the following theorems we characterize the $p$-adic valuations of Lucas atoms, dealing with the cases $p \nmid t$ and $p$ divides both $s$ and $t$ (note that when $p\nmid s$ and $p\mid t$, the $p$-adic valuation of Lucas polynomials is always zero).

\begin{Theorem} \label{thm:patom}
Let $p\geq 3$ be a prime number such that and $k=\rho(p, U)$. Let us suppose that $p\nmid t$. Then
\begin{align*}
v_p(P_n)=\begin{cases} v_p(U_k) \ \  &\text{if} \ \  n=k \\
1 \ \  &\text{if} \ \  n=kp^h, \ h\geq 1 \\ 
0 \ \  &\text{otherwise.}
\end{cases}
\end{align*}
\end{Theorem}
\begin{proof}
First, consider $p \nmid \Delta$ and we prove by induction that
\begin{equation} \label{eq:vP} \begin{cases} v_p(P_n) \geq 1, \quad \text{if } n = k p^h, \ h \geq 0  \cr v_p(P_n) = 0, \quad \text{otherwise}.\end{cases} \end{equation}
By Lemmas \ref{Lem: PropRank} and \ref{Lem: RankP} we know that $v_p(P_k) = v_p(U_k) \geq 1$ and $v_p(P_n) = 0$ for all $n < k$. Now, fixed a certain positive integer $n$, we suppose that \eqref{eq:vP} holds for all $i < n$ and we prove that it holds also for $n$. By \eqref{eq:fatt}, we have
\[ v_p(P_n) = v_p(U_n) - \sum\limits_{\substack{d\mid n\\ d\not=n}} v_p(P_d). \]
If $k \nmid n$, then $v_p(P_n) = 0$, since $v_p(U_n) = 0$ by Lemma \ref{Lem: PropRank}.\\ 
If $k \mid  n$ and $v_p(n) = 0$, then
\[ v_p(P_n) = v_p(U_k) + v_p(n) - \sum\limits_{\substack{d\mid n\\ d\not=n}} v_p(P_d), \]
by Theorem \ref{Thm: Carlo}. Thus,
\[ v_p(P_n) = - \sum\limits_{\substack{d\mid n\\ d\not\in\{k,n\}}} v_p(P_d) = 0, \]
since $v_p(P_k) = v_p(U_k)$ by Lemma \ref{Lem: RankP} and all divisors $d$ of $n$ can not be of the form $kp^h$ for some $h \geq 0$. \\
If $k \mid  n$ and $n = kmp^h$, for some positive integer $h$ and $m \not=1$ with $\gcd(m, p) = 1$, then
\[ v_p(P_n) = v_p(n) - \sum\limits_{\substack{d\mid n\\ d\not\in\{k,n\}}} v_p(P_d) = h - \sum_{i = 1}^h v_p(P_{kp^i}), \]
since, by Lemma \ref{lemma:leg}, $p \nmid k$ Moreover, by inductive hypothesis, $v_p(P_{kp^i}) \geq 1$, for all $1 \leq i \leq h$, and $v_p(P_d) = 0$ otherwise. Since the $p$-adic valuations of Lucas atoms can not be negative, we must have that $v_p(P_{kp^i}) = 1$, for all $1 \leq i \leq h$, and thus $v_p(P_n) = 0$. \\
Finally, if $k \mid  n$ and $n = kp^h$ for a positive integer $h$, then
\[v_p(P_n) = v_p(n) - \sum\limits_{\substack{d\mid n\\ d\not\in\{k,n\}}} v_p(P_d) = h - \sum_{i=1}^{h-1} v_p(P_{kp^i}) = 1.\]
Now we consider the case $p \mid  \Delta$ and we prove by induction that \eqref{eq:vP} still holds. Considering that in this case $k = p$ by Lemma \ref{lemma:leg} and $v_p(U_n) = 0$ for all $n < p$ by Theorem \ref{Thm: Carlo}, then from Lemmas \ref{Lem: PropRank} and \ref{Lem: RankP}, the base step is true. Now, consider \eqref{eq:vP} true for all $i < n$, for a fixed positive integer $n$, and we prove that it holds also for $n$. \\
If $p \nmid n$, then $v_p(P_n) = v_p(U_n) - \sum\limits_{\substack{d\mid n\\ d\not=n}} v_p(P_d) = 0$. \\
If $p \mid  n$ and $n = mp^h$, for an integer $m > 1$ such that $\gcd(m, p) = 1$, then, using the inductive hypothesis and Theorem \ref{Thm: Carlo}, we obtain
\begin{align*}
v_p(P_n) &= v_p(U_n) - \sum\limits_{\substack{d\mid n\\ d\not=n}} v_p(P_d) = v_p(n) +v_p(U_p) - 1 -v_p(P_p) - \sum\limits_{\substack{d\mid n\\ d\not\in\{p,n\}}} v_p(P_d) \\
&= h - 1 - \sum_{i=2}^h v_p(P_{p^i}).
\end{align*}
Since $v_p(P_{p^i}) \geq 1$, and the $p$-adic valuation of Lucas atom is nonnegative, we must have $v_p(P_{p^i}) = 1$, for all $i \geq 2$ and thus $v_p(P_n) = 0$.\\
If $p\mid n$ and $n = p^h$, then 
\begin{align*}v_p(P_n) &= v_p(U_n) - \sum\limits_{\substack{d\mid n\\ d\not=n}} v_p(P_d) = v_p(n) + v_p(U_p) - 1 - v_p(P_p) - \sum\limits_{\substack{d\mid n\\ d\not\in\{p,n\}}} v_p(P_d) \\
&= h - 1 - \sum_{i=2}^{h-1} v_p(P^i) = 1. 
\end{align*}
\end{proof}

\begin{Theorem}\label{thm:patom2}
If $2 \nmid t$ and $2\mid s$ (i.e., $2\mid \Delta$ and $\rho(2,U) = 2$), then
\begin{align*}
v_2(P_n) =
\begin{cases}
v_2(U_2) \ \ &  \text{if } n = 2 \cr
1 & \text{if } n = 2^h, \ h \geq 2 \cr
0, & \text{otherwise}.
\end{cases}   
\end{align*}
If $2 \nmid t$ and $2\nmid s$ (i.e., $2\nmid\Delta$ and $\rho(2,U) = 3$), then
\begin{align*}
v_2(P_n) = 
\begin{cases}
v_2(U_3)  \ \ & \text{if } n = 3 \cr
v_2(U_6) - v_2(U_3) & \text{if } n = 6 \cr
1 & \text{if } n = 3 \cdot 2^h, \ h \geq 2 \cr
0 & \text{otherwise}.
\end{cases}
\end{align*}
\end{Theorem}
\begin{proof}
The proof for the case $2 \nmid t$ and $2\mid s$ follows as in the proof of Theorem \ref{thm:patom}.

If $2 \nmid st$, we have $v_2(P_3) = v_2(U_3) \not=0$ and $v_2(P_6) = v_2(U_6) - v_2(U_3)$, since $v_2(P_2) = 0$. Then, the other cases follow by induction as in the proof of Theorem \ref{thm:patom}. 
\end{proof}

\begin{Theorem}
Let $p$ be a prime such that $s = p^a s'$ and $t = p^b t'$, where $p \nmid s' t'$, $a, b >0$ with $b\geq 2a$ possibly infinite and $a$ finite. Then for all $n \geq 2$, we have
\begin{align*}
 v_p(P_n) = \begin{cases}  \varphi(n)a \ \ & \text{if } b > 2a \cr
 \varphi(n)a + v_p(P'_n) & \text{if } b = 2a,
 \end{cases}
 \end{align*}

where $(P'_n)_{n \geq 1}$ is the sequence of Lucas atoms associated to the Lucas sequence $(U'_n)_{n \geq 0}$ with characteristic polynomial $X^2 - s' X - t'$. 
\end{Theorem}
\begin{proof}
We prove the statement by induction. If $b > 2a$, the base step is straightforward. Now suppose $v_p(P_i) = \varphi(i) a$, for all $i < n$ and then
\[v_p(P_n) = v_p(U_n) - \sum\limits_{\substack{d\mid n\\ d\not=n}} v_p(P_d) = (n-1)a - \sum\limits_{\substack{d\mid n\\ d\not\in\{1,n\}}} \varphi(d)a = a \varphi(n),\]
where we exploited Theorem \ref{Thm: Ballot} and $v_p(P_1) = 0$.\\
Also for $b = 2a$ the base step is straightforward and supposing that the thesis is true for all the nonnegative integers less than $n$, we have
\[ v_p(P_n) = (n-1) a +v_p(U'_n) - \sum\limits_{\substack{d\mid n\\ d\not=n}}( a \varphi(d) + v_p(P'_d)) = a \varphi(n) + v_p(P'_n) \]
since $v_p(P'_n)  = v_p(U'_n) - \sum\limits_{\substack{d\mid n\\ d\not=n}} v_p(P'_d)$.
\end{proof}

\begin{Theorem}\label{thm:patom4}
Let $p$ be a prime such that $s = p^a s'$ and $t = p^b t'$, where $a, b>0$ and $p \nmid s' t'$ and $b < 2a$ with $a$ possibly infinite. If $p \not\in \{2, 3\}$ or $b < 2a + 1$, then for each $n\geq 2$, we have
\begin{align*}
v_p(P_n) = \begin{cases} a \ \ & \text{if } n = 2 \cr 
b \cfrac{\varphi(n)}{2} + 1 \ & \text{if } n = 2p^h, \ h \geq 1 \cr
b \cfrac{\varphi(n)}{2} & \text{otherwise}.
\end{cases}
\end{align*}
\end{Theorem}
\begin{proof}
We proceed by induction. The basis is straightforward. Now, consider $n = 2p^h$ for some $h \geq 1$ and that the thesis is true for all $i < n$. Then,
\[v_p(P_n) = v_p(U_n) - \sum\limits_{\substack{d\mid n\\ d\not=n}} v_p(P_d).\]
By Theorem \ref{Thm: Ballot}, we have
\[v_p(P_n) = bp^h  + (a-b) + h - \sum\limits_{\substack{d\mid n\\ d\not=n}} v_p(P_d)\]
and using the inductive hypothesis we get
\begin{align*} v_p(P_n) &= bp^h  + (a-b) + h - v_p(P_2) - \cfrac{b}{2} \sum_{i=1}^h \varphi(p^j) - \sum_{i=1}^{h-1}\left( b\cfrac{\varphi(2p^i)}{2} + 1 \right)  \\
&= bp^h -b + 1 - \cfrac{b}{2}\left( \sum_{i=1}^h \varphi(p^i) + \sum_{i=1}^{h-1} \varphi(2p^i) \right) \\
&= bp^h - b + 1 - \cfrac{b}{2} ( 2p^h-1 - \varphi(2p^h) - \varphi(2) ) = b \cfrac{\varphi(n)}{2} + 1.
\end{align*}
The other cases, when $n \not= 2p^h$, follow in a similar way
\end{proof}

Finally, the next two theorems fully complete our analysis. The techniques of the proofs are similar to the previous ones and they exploit the results of Ballot.

\begin{Theorem}
If $s = 3^as'$ and $t= 3^bt'$, with $a, b \in\mathbb N_+$, $3 \nmid s't'$ and $b= 2a - 1$, then for each $n\geq 2$, we have 
\begin{align*}
v_3(P_n) = \begin{cases} a \ \ & \text{if } n = 2 \cr 
b + 1 +v_3(s'^2+t') \ \ & \text{if } n = 6 \cr
3^{h-1}b 
+ 1 & \text{if } n = 2\cdot 3^h, \ h \geq 2 \cr
b \cfrac{\varphi(n)}{2} & \text{otherwise}.  \end{cases}
\end{align*}
\end{Theorem}

\begin{Theorem}
If $s = 2^as'$ and $t= 2^bt'$, with $a, b \in\mathbb N_+$, $2 \nmid s't'$ and $b= 2a - 1$, then for each $n\geq 2$, we have
\begin{align*}
v_2(P_n) = \begin{cases} a \ \ & \text{if } n = 2 \cr 
b + 1 +v_2(s'^2+t') \ \ & \text{if } n = 4 \cr 
2^{h-2}b+1 & \text{if } n = 2^h, \ h \geq 3  \cr b \cfrac{\varphi(n)}{2} & \text{otherwise}. \end{cases}
\end{align*}
\end{Theorem}

\bigskip

\subsection{Another approach for the study of the $p$--adic valuations of Lucas atoms}

From \eqref{eq:fatt} we derive
\begin{equation*}
    \nu_p(U_n)=\sum_{d\mid n} \nu_p(P_d)
\end{equation*}
for each prime number $p$ and positive integer $n$. Hence, by M\"{o}bius transformation formula we obtain
\begin{equation*}
    \nu_p(P_n)=\sum_{d\mid n} \mu(d)\nu_p(P_{\frac{n}{d}}), \quad n\geq 1,
\end{equation*}
where $\mu$ denotes the M\"{o}bius function. From Theorems \ref{Thm: Carlo}, \ref{Thm: Carlo2}, \ref{Thm: Ballot} we know that the sequence $(\nu_p(U_n))_{n\geq 1}$ is a linear combination of identity function, characteristic functions of some arithmetic progressions and products of functions of this form with the sequence of $p$-adic valuations of consecutive positive integers. By the bilinearity of Dirichlet convolution we can express $p$-adic valuations of Lucas atoms as linear combinations of the transformations of the mentioned functions via Dirichlet convolution with M\"{o}bius function.

For each integer $r\geq 1$ we denote by $\mathbf 1_r$ the characteristic function of the multiplicities of $r$, i.e. the function given by the formula
$$\mathbf 1_r(n)=
\begin{cases}
    1 \ \ &\text{if} \ \ r\mid n \cr
    0 \ \ &\text{if} \ \ r\mid n.
\end{cases}$$

Now we are ready to state the crucial lemma.

\begin{Lemma}\label{Lem: conv}
    Let $p$ be a prime number and $q,r\geq 1$ be integers, where $\gcd(q,p)=1$. Then, for each integer $n\geq 1$ we have:
    \begin{enumerate}
        \item $\sum_{d\mid n}\mu(d)\cdot\frac{n}{d}=\varphi(n)$;
        \item $\sum_{d\mid n}\mu(d)\cdot\mathbf 1_r\left(\frac{n}{d}\right)=
        \begin{cases}
            1 \ \ &\text{if} \ \ n=r \cr
            0 \ \ &\text{if} \ \ n\neq r;
        \end{cases}$
        \item $\sum_{d\mid n}\mu(d)\cdot\mathbf 1_{q}\left(\frac{n}{d}\right)\nu_p\left(\frac{n}{d}\right)=
        \begin{cases}
            1 \ \ &\text{if} \ \ n=p^hq \ \ \text{for some} \ \ h\in\mathbb N_+ \cr
            0 \ \ &\text{otherwise}.
        \end{cases}$
    \end{enumerate}
\end{Lemma}

\begin{proof}
    The first identity follows from application of M\"{o}bius transformation formula to classical identity
    $$\sum_{d\mid n}\varphi (d)=n.$$

    We start the proof of remaining identities with the note that their left hand sides vanish when $n$ is not divisible by $r$ ($q$, respectively) as all the divisors of $n$ are not  divisible by $r$ ($q$, respectively). From this moment on, we consider the case of $n$ divisible by $r$ ($q$, respectively).

    Write $n=rn'$ for some $n'\in\mathbb N_+$. Then,
    $$
    \sum_{d\mid n}\mu(d)\cdot\mathbf 1_r\left(\frac{n}{d}\right)=\sum_{d\mid n}\mu\left(\frac{n}{d}\right)\cdot\mathbf 1_r(d)=\sum_{r\mid d\mid n}\mu\left(\frac{n}{d}\right)=\sum_{d'\mid n'}\mu\left(\frac{n'}{d'}\right)=
        \begin{cases}
            1 \ \ &\text{if} \ \ n'=1 \cr
            0 \ \ &\text{if} \ \ n'\neq 1,
        \end{cases}
    $$
    where we write $d=rd'$ for $d$ divisible by $r$ and the last equality is a well known fact. The second identity is proved.

    Similarly we start the proof of the last identity when $n=qn'$, $n'\in\mathbb N_+$. Write
    \begin{equation}\label{eq: conv}
    \begin{split}
      &\sum_{d\mid n}\mu(d)\cdot\mathbf 1_q\left(\frac{n}{d}\right)\nu_p\left(\frac{n}{d}\right)=\sum_{d\mid n}\mu\left(\frac{n}{d}\right)\cdot\mathbf 1_q(d)\nu_p(d)=\sum_{q\mid d\mid n}\mu\left(\frac{n}{d}\right)\nu_p(d)\\
      &=\sum_{d'\mid n'}\mu\left(\frac{n'}{d'}\right)\nu_p(d')=\sum_{d'\mid n'}\mu(d')\nu_p\left(\frac{n'}{d'}\right),  
    \end{split}
    \end{equation}
    where we write $d=qd'$ for $d$ divisible by $q$ and use the coprimality of $q$ and $p$ in the penultimate equality. If $n'=p^h$, $h\in\mathbb N_+$, then the last expression in \eqref{eq: conv} takes the form
    $$\nu_p(p^h)-\nu_p(p^{h-1})=1.$$
    Otherwise, $n'=p^huw$, where $h,w\in\mathbb N_+$, $u$ is a prime number, and $p\nmid uw$. Then we may write the last expression in \eqref{eq: conv} as follows.
    $$\sum_{d'\mid pw}\left(\mu(d')\nu_p\left(\frac{n'}{d'}\right)+\mu(ud')\nu_p\left(\frac{n'}{ud'}\right)\right)=\sum_{d'\mid pw}\left(\mu(d')\nu_p\left(\frac{n'}{d'}\right)-\mu(d')\nu_p\left(\frac{n'}{d'}\right)\right)=0$$
    The proof of the last identity is finished.
\end{proof}

Now, we can reformulate Theorems \ref{Thm: Carlo}, \ref{Thm: Carlo2}, \ref{Thm: Ballot}.

\begin{Theorem}[Reformulation of Theorem \ref{Thm: Carlo}]\label{Thm: Carlo ref}
Let $p\geq 3$ be a prime number such that $p\nmid t$ and $k=\rho(p, U)$. Then,
\begin{align*}
v_p(U_n)=\begin{cases} v_p(n) +  (v_p(U_p) - 1)\mathbf 1_p(n) \ \  &\text{if} \ \  p\mid \Delta\\
\mathbf 1_k(n)v_p(n) + v_p(U_k)\mathbf 1_k(n) \ \  &\text{if} \ \  p\nmid \Delta,
\end{cases}
\end{align*}
for each positive integer $n$, where $\nu_p(U_p)=1$ for $p\geq 5$ if $p\mid\Delta$.
\end{Theorem}

\begin{Theorem}[Reformulation of Theorem \ref{Thm: Carlo2}]\label{Thm: Carlo2 ref}
If $2 \nmid t$ and $2\mid s$ (i.e., $2\mid \Delta$ and $\rho(2,U) = 2$), then
\[v_2(U_n) = v_2(n) +  (v_2(U_2) - 1)\mathbf 1_2(n).\]
If $2 \nmid t$ and $2\nmid s$ (i.e., $2\nmid\Delta$ and $\rho(2,U) = 3$), then
\[v_2(U_n) = v_2(n)\mathbf 1_3(n) + v_2(U_3)\mathbf 1_3(n) + (v_2(U_6) - v_2(U_3) - 1)\mathbf 1_6(n).\]
\end{Theorem}

\begin{Theorem}[Reformulation of Theorem \ref{Thm: Ballot}]\label{Thm: Ballot ref}
Let $p$ be a prime such that $s = p^a s'$ and $t = p^b t'$, where $p \nmid s' t'$ and $a,b\in\mathbb N_+\cup\{\infty\}$ with $a,b$ necessarily finite in the case of $b=2a$. Then for all $n \geq 1$, we have
\[ v_p(U_n) = \begin{cases}  an - a \ \ &\text{if } \ \ b > 2a \cr an - a + v_p(U'_n) \ \ &\text{if } \ \ b = 2a, \end{cases} \]
where $(U'_n)_{n \geq 0}$ is the Lucas sequence with characteristic polynomial $X^2 - s' X - t'$. For $b < 2a$, we have that
\begin{align*}
    v_p(U_{n}) = & \frac{b}{2}\cdot n - \frac{b}{2} + \left(a-\frac{b}{2}\right)\mathbf 1_2(n) + v_p\left(\frac{n}{2}\right)\mathbf 1_2(n) + \lambda\mathbf 1_{2p}(n)\\
    = & \frac{b}{2}\cdot n - \frac{b}{2} + \left(a-\frac{b}{2}\right)\mathbf 1_2(n) + v_p(n)\mathbf 1_2(n) - \delta_{p,2}\mathbf 1_2(n) + \lambda\mathbf 1_{2p}(n),
\end{align*}
where
\[ \lambda = \begin{cases} v_p(s'^2 -t') \ \ &\text{if } 2 \leq p \leq 3, \ 2a = b + 1 \cr 0 \ \ &\text{otherwise}  \end{cases} \]
and
\[ \delta_{p,2} = \begin{cases} 1 \ \ &\text{if} \ \ p=2 \cr 0 \ \ &\text{if} \ \ p\neq 2. \end{cases} \]
\end{Theorem}

Applying Lemma \ref{Lem: conv} to the above theorems and using bilinearity of Dirichlet convolution it is possible to obtain the main results of this section, exploiting this different approach.

\section{Non-holonomicity of the sequence of Lucas atoms} \label{sec:nonrec}

In the striking opposition to the Lucas sequence, which is binary linearly recurrent from its definition, the sequence of Lucas atoms evaluated at any pair of integer values of variables $s$ and $t\neq 0$ is not even holonomic, i.e. polynomially recurrent. This follows from the more general fact stated below.

\begin{Theorem}\label{Thm: nonhol}
    For each integers $s,t$ with $t\neq 0$ there do not exist $l\in\mathbb N_+$ and\linebreak $G_0(X),G_1(X),\ldots ,G_l(X)\in\mathbb Q[X]$ such that
    $$P_n(s,t)=G_0(n)+\sum_{j=1}^l G_j(n)P_{n-j}(s,t)$$
    for sufficiently large $n\in\mathbb N_+$.
\end{Theorem}

\begin{proof}
    Fix non-zero values of $s$ and $t$ and assume that
    $$P_n(s,t)=G_0(n)+\sum_{j=1}^l G_j(n)P_{n-j}(s,t)$$
    for some $l,n_0\in\mathbb N_+$ and $G_1(X),\ldots ,G_l(X)\in\mathbb Q[X]$ and all $n\geq n_0$. Choose a prime number $p$ so large that $p$ divides neither of $s$, $t$ and the denominators of the coefficients of $G_0(X),G_1(X),\ldots ,G_l(X)$ written in the irreducible form. Then $G_j(n_1)\equiv G_j(n_2)\pmod{p}$ for each $j\in\{0,1,\ldots ,l\}$ and integers $n_1\equiv n_2\pmod{p}$. Since the set of $l$-tuples 
    $$\{(P_{mp-1}(s,t),\ldots ,P_{mp-l}(s,t))\pmod{p}: mp\geq n_0\}$$
    is finite, by pigeon hole principle we find integers $m_2>m_1\geq\frac{n_0}{p}$ such that 
    $$(P_{m_1p-1}(s,t),\ldots ,P_{m_1p-l}(s,t))\equiv (P_{m_2p-1}(s,t),\ldots ,P_{m_2p-l}(s,t))\pmod{p}.$$
    Putting $m_0=m_2-m_1$ we show by easy induction that $P_{n+m_0p}(s,t)\equiv P_n(s,t)\pmod{p}$ for each integer $n\geq m_1p-l$. In particular, the set of indices $n$ such that $p\mid P_n(s,t)$ is a finite union of infinite arithmetic progressions and a finite set. Howerver, Theorems \ref{thm:patom}, \ref{thm:patom2} and \ref{thm:patom4} show that is not the case. The contradiction proves non-holonomicity of the sequence $(P_n(s,t))_{n\geq 1}$.
\end{proof}

As a direct consequence we get the following.

\begin{Corollary}
    There do not exist $l\in\mathbb N_+$ and $G_0(s,t,X),G_1(s,t,X),\ldots ,G_l(s,t,X)\in\mathbb Q[s,t,X]$ such that
    $$P_n(s,t)=G_0(s,t,n)+\sum_{j=1}^l G_j(s,t,n)P_{n-j}(s,t)$$
    for sufficiently large $n\in\mathbb N_+$.
\end{Corollary}

Since a lot of number sequences having combinatorial interpretation are polynomially recurrent, Theorem \ref{Thm: nonhol} suggests us that the problem of finding some natural combinatorial interpretation of Lucas atoms seems to be intractable.

Let us notice that the proof of Theorem \ref{Thm: nonhol} is valid only if $t\neq 0$. From the definition of Lucas atoms we have $P_1(s,0)=1$ and $P_n(s,0)=s^{\varphi(n)}$ for $n\geq 2$. Then it is quite interesting to check for which value of $s$ the sequence $(P_n(s,0))_{n\geq 1}$ is (polynomially) recurrent.

\begin{Conjecture}
    Let $s\in\mathbb Z$. Then the sequence $(P_n(s,0))_{n\geq 1}$ is polynomially recurrent if and only if $s\in\{-1,0,1\}$.
\end{Conjecture}

The ``if" part of the above conjecture is obvious. Its burden is contained in the ``only if" part.

\section*{Acknowledgments}
We would like to thank Carlo Sanna for making us aware of the Stewart paper \cite{STE1} and Christian Ballot for providing us the errata corrige of his paper.\\
The first, the third and the fourth author are members of GNSAGA of INdAM. The research of the second author is supported by the Grant of the Polish National Science Centre No. UMO-2019/34/E/ST1/00094. The research cooperation was funded by the program Excellence Initiative – Research University at the Jagiellonian University in Kraków.

\end{document}